\documentclass[10pt,bezier]{article}
\usepackage{amsmath,amssymb,amsfonts,euscript,graphicx}

\newtheorem{prethm}{{\bf Theorem}}

\newenvironment{thm}{\begin{prethm}{\hspace{-0.5
               em}{\bf.}}}{\end{prethm}}

\newtheorem{prepro}[prethm]{{\bf Theorem}}

\newtheorem{preprop}[prethm]{{\bf Proposition}}

\newtheorem{precor}[prethm]{{\bf Corollary}}

\newtheorem{preconj}[prethm]{{\bf Conjecture}}

\newtheorem{preremark}[prethm]{{\bf Remark}}

\newtheorem{preexample}[prethm]{{\bf Example}}

\newtheorem{prelem}[prethm]{{\bf Lemma}}

\newtheorem{prelam}{{\bf Lemma}}

\newtheorem{preproof}{{\bf Proof.}}

\newenvironment{proof}[1]{\begin{preproof}{\rm
               #1}\hfill{$\Box$}}{\end{preproof}}

\title{\bf \large A Note on Co-Maximal Ideal Graph of Commutative Rings
\thanks
{{\it Key Words}:  Co-maximal ideal graph, Star graph.}
\thanks {2010{ \it Mathematics Subject Classification}: 05C10, 05C25,  05C69, 16D25.
 }}

\author{{\normalsize  {\sc S. Akbari${}^{\mathsf{a,c}}$}, {\sc B.
Miraftab${}^{\mathsf{a}}$}, {\sc R. Nikandish${}^{\mathsf{b,c}}$}} \\
 {\footnotesize{${}^{\mathsf{a}}$\it Department of Mathematical
 Sciences, Sharif
University of Technology, Tehran, Iran}}\\
{\footnotesize{${}^{\mathsf{b}}$\it Department of Mathematics, Jundi-Shapur University of Technology,  Dezful, Iran}}\\
{\footnotesize{\rm P.O. Box 64615-334}}\\
{\footnotesize{${}^{\mathsf{c}}$\it School of Mathematics, Institute
for Research in Fundamental Sciences, \rm{(IPM)}}}\\
{\footnotesize{\rm P.O. Box 19395-5746}}\\
{\footnotesize{}}\\
{\footnotesize{$\mathsf{s\_akbari@sharif.edu}$\quad\quad
$\mathsf{babak.math@gmail.com}$\quad\quad$\mathsf{r.nikandish@ipm.ir}$}}}

\date{}
\usepackage{color}
\begin{document}
\maketitle
\begin{abstract}
{\small Let $R$ be a commutative ring with unity. The co-maximal ideal graph of $R$, denoted by $\Gamma(R)$, is a
graph whose vertices are the proper ideals of $R$ which are not contained in the Jacobson radical of
$R$, and two vertices $I_1$ and $I_2$ are adjacent if and only if $I_1 + I_2 = R$.
We classify all commutative rings whose co-maximal ideal graphs are planar.
In 2012 the following question was posed: If $\Gamma(R)$ is an infinite star graph,
can $R$ be isomorphic to the direct product of a field and a local ring? In this paper, we give an affirmative answer to this question. }
\end{abstract}

\vspace{9mm} \noindent{\bf\large 1. Introduction}\\\\
{\indent When one assigns a graph to
an algebraic structure numerous interesting algebraic problems
arise from the translation of some graph-theoretic parameters
such as clique number, chromatic number, independence number and
so on. There are a lot of papers
which apply combinatorial methods to obtain algebraic results, for instance see \cite {ak2}, \cite {akn1}, \cite {akn2}, \cite {sharm} and \cite {ye}.\\

\noindent Let $G$ be a graph with the vertex set $V(G)$.
A bipartite graph with part sizes $m$ and $n$ is denoted by $K_{m,n}$. If the size of one of the parts is $1$, then the graph is said to
be a \textit{star graph}. A \textit{clique} of $G$ is a
complete subgraph of $G$ and the number of vertices in a
largest clique of $G$, denoted by $\omega(G)$, is called the
\textit{clique number} of $G$.
An \textit{independent set} of $G$ is a subset of the
vertices of $G$ such that no two vertices in the subset represent
an edge of $G$. The \textit{independence number} of $G$, denoted
by $\alpha(G)$, is the cardinality of the largest independent set. A graph is said to be \textit{ planar}, if it can be drawn in
the plane so that its edges intersect only at their ends.\\
\noindent Throughout this paper  $R$ is a commutative ring with unity.  The set of maximal ideals of $R$ and the Jacobson radical of $R$ are denoted by ${\rm Max}(R)$ and ${\rm J}(R)$, respectively. The ring $R$ is called \textit{local} if $|{\rm Max}(R)|=1$. The ring $R$ is said to be  \textit{uniserial} if  ideals of $R$ are totally ordered by inclusion.\\
\noindent The \textit{co-maximal ideal graph} of $R$, denoted by $\Gamma(R)$, is a
graph whose vertices are the proper ideals of $R$ which are not contained in the Jacobson radical of
$R$, and two vertices $I_1$ and $I_2$ are adjacent if and only if $I_1 + I_2 = R$. This graph was first introduced and studied in \cite{ye}. In 2012, Ye and Wu in \cite[Question 4.12]{ye} asked the following question: If $\Gamma(R)$ is an infinite star graph, can $R$ be isomorphic to the direct
product of a field and a local ring? In this paper, we give an affirmative answer to this question. Indeed, we show that there exists a vertex of  $\Gamma(R)$ which is adjacent to all other vertices if and only if $R$ is isomorphic to the direct
product of a local ring and a field. Also we characterize all commutative rings whose co-maximal ideal graphs are planar.
}

\vspace{9mm} \noindent{\bf\large 2. Results}\\

\noindent In this section, we classify all rings whose co-maximal ideal graphs have a vertex which is adjacent to all other vertices. We start with the following theorem.
\begin{thm}\label{1}
Let $R$ be a ring. Then there exists a vertex of  $\Gamma(R)$ which is adjacent to all other vertices if and only if $R$ is isomorphic to the direct
product of a local ring and a field.
\end{thm}
\begin{proof}
{One side is clear. For the other side, let $I$ be a vertex adjacent to all other vertices and $a\in I\setminus {\rm J}(R)$. Since $I$ is adjacent to all other vertices, we deduce that $I=Ra$ and $I$ is a maximal ideal of $R$. Also, $Ra^2$ is a vertex of $\Gamma(R)$ and so $Ra=Ra^2$. Thus $a=ta^2$, for some $t\in R$. Clearly, $1\neq 1-ta$ is  a non-zero idempotent. By \cite[Proposition 5.10]{ander}, $R\cong R_1\times R_2$, for some rings $R_1$ and $R_2$. We show that at least one of the rings  $R_1$ and $R_2$ is a field. With no loss of generality, we may assume that $I=R_1\times \mathfrak{m}$, where $\mathfrak{m}$ is a maximal ideal of $R_2$. Obviously, if $\mathfrak{m}\neq 0$, then $I$ is not adjacent to $R_1\times 0$, a contradiction. Thus $R_2$ is a field. Now, we prove that $R_1$ is a local ring. By contrary, assume that $R_1$ is not a local ring. Thus there exists an ideal of $R$, say $J=\mathfrak{m}_1\times 0$, where $\mathfrak{m}_1$ is a maximal ideal of $R_1$, and $J$ is a vertex of $\Gamma(R)$. But $I$ and $J$ are not adjacent, a contradiction and the proof is complete.
}
\end{proof}
\noindent In the sequel of this paper, we provide some conditions under which $\Gamma(R)$ is a finite graph.
\begin{thm}\label{independence}
If  $\alpha(\Gamma(R))<\infty$, then $\Gamma(R)$ is a finite graph.
\end{thm}
\begin{proof}
{Since $\alpha(\Gamma(R))<\infty$, we deduce that $\alpha(\Gamma(\frac{R}{{\rm J}(R)}))<\infty$. Thus $\frac{R}{{\rm J}(R)}$ is an Artinian ring and so by \cite[Theorem 8.7]{ati},  $\frac{R}{{\rm J}(R)}$ has finitely many maximal ideals. Therefore, $|{\rm Max}(R)|<\infty$ and hence by \cite[Theorem 3.1]{ye}, $\omega(\Gamma(R))<\infty$.  Now, the result follows from Ramsey's Theorem, see \cite[Theorem 12.5]{bondy}.
}
\end{proof}

\begin{thm}\label{finit}
If each vertex of $\Gamma(R)$ has a finite degree, then $R$ has finitely many ideals. Moreover, $R$ is a direct product of finitely many uniserial rings and a finite ring.
\end{thm}
\begin{proof}{
Let $I$ be a vertex of $\Gamma(R)$. Then there exists an ideal $L$ of $R$ such that $I+L=R$. So there exist two elements $a\in I$ and $b\in L$ such that $a+b=1$. Since $1=(a+b)^n=\Sigma_{k=0}^n {n\choose k}a^kb^{n-k}$, we conclude that $I+Rb^n=R$, for $n=1,2,\ldots$. Furthermore since  $I$ has finite degree and $Rb^n\nsubseteq {\rm J}(R)$, for $n=1,2,\ldots$, we find that  $(b^t)=(b^{2t})$, for some $t\geq 1$. Therefore $b^t=lb^{2t}$, for some $l\in R$, and so $1-lb^t$ is a non-trivial idempotent element. Hence $R\cong R_1\times R_2$, for some rings $R_1$ and $R_2$, see \cite[Proposition 5.10]{ander}. We show that $R_i$ contains finitely many ideals for $i=1,2$. If $\{I_i\}_{i=1}^{\infty}$ is an infinite family of ideals of $R_1$, then the vertex $J=R_1\times 0$ is adjacent to $I_i\times R_2$, for $i\geq 1$. So the degree of $J$ is not finite, a contradiction. By a similar argument $R_2$ has finitely many ideals. Thus $R$ contains finitely many ideals. It follows from \cite[Theorem 2.4]{yas}, $R$ is a direct product of finitely many uniserial rings and a finite ring.
}\end{proof}
\noindent To prove the next result, we need a celebrated theorem due to Kuratowski.
\begin{thm} {\rm \cite [Theorem 10.30]{bondy}}
A graph is planar if and only if it contains no subdivision of either $K_5$ or $K_{3,3}$.
\end{thm}
\noindent We close this paper with the following theorem.
\begin{thm}\label{ind}
Let $\Gamma(R)$ be a finite graph. If $\Gamma(R)$ is planar, then one of the following holds:\\

\noindent {\rm(i)} $R\cong R_1\times R_2$, where $R_1$ and $R_2$ are local rings and one of $R_i$ has at most three ideals.\\
\noindent {\rm(ii)} $R\cong R_1\times R_2\times R_3$ and each $R_i$ has at most one non-trivial ideal.
\end{thm}
\begin{proof}{Assume that $\Gamma(R)$ is planar. Since $\Gamma(R)$ is finite, it follows from Theorem \ref{finit} that $R$ is an Artinian ring. By \cite[Theorem 8.7]{ati}, $R\cong R_1\times \cdots \times R_n$, where $R_i$ is an Artinian local ring, for $i=1,\ldots,n$. Now, Kuratowski's Theorem implies that $|{\rm Max}(R)|\leq 4$. Assume that $|{\rm Max}(R)|=4$ and ${\rm Max}(R)=\{\mathfrak{m}_1,\mathfrak{m}_2,\mathfrak{m}_3,\mathfrak{m}_4\}$. It is not hard to see that $V_1=\{\mathfrak{m}_1,\mathfrak{m}_2,\mathfrak{m}_1\mathfrak{m}_2\}$ and $V_2=\{\mathfrak{m}_3,\mathfrak{m}_4,\mathfrak{m}_3\mathfrak{m}_4\}$ induce $K_{3,3}$. Thus $|{\rm Max}(R)|\leq 3$. If $|{\rm Max}(R)|=2$, then (i) is directly follows from Kuratowski's Theorem and \cite[Theorem 4.5]{ye}.
Hence suppose that $|{\rm Max}(R)|= 3$ and so $R\cong R_1\times R_2\times R_3$. With no loss of generality, assume that $R_1$ has at least two non-trivial ideals $I$ and $J$. Thus two sets $V_1=\{I\times R_2\times R_3, J\times R_2\times R_3, 0\times R_2\times R_3 \}$ and $V_2=\{R_1\times 0 \times R_3, R_1\times R_2\times 0, R_1\times 0\times 0 \}$ imply that $\Gamma(R)$ contains $K_{3,3}$, a contradiction. Therefore, each $R_i$ has at most one non-trivial ideal. This completes the proof.
}
\end{proof}
%
%
\noindent{\bf Acknowledgements.}  The research
of the first and the third authors were in part supported by a grant from the IPM No.
92050212, and No. 92050017, respectively.


{}

\end{document}